\documentclass[12pt, reqno]{amsart}
\usepackage{amsmath, amsthm, amscd, amsfonts, amssymb, graphicx, color, mathrsfs}
\usepackage[bookmarksnumbered, colorlinks, plainpages]{hyperref}

\newtheorem{theorem}{Theorem}[section]
\newtheorem{lemma}[theorem]{Lemma}
\newtheorem{proposition}[theorem]{Proposition}
\newtheorem{corollary}[theorem]{Corollary}
\theoremstyle{definition}

\theoremstyle{remark}
\newtheorem{remark}[theorem]{Remark}
\numberwithin{equation}{section}

\begin{document}
\setcounter{page}{1}

\title[boundedness of periodic pseudo-differ. operators ]{On the boundedness of periodic pseudo-differential operators}

\author[Duv\'an Cardona]{Duv\'an Cardona$^1$}

\address{$^{1}$ Department of Mathematics, Universidad de los Andes,Bogot\'a- Colombia.}
\email{\textcolor[rgb]{0.00,0.00,0.84}{d.cardona@uniandes.edu.co;
duvanc306@gmail.com}}



\date{Received:;  Revised:; Accepted:;
\newline \indent $^{1}$ Universidad de los Andes, Bogot\'a - Colombia}

\begin{abstract}
In this paper we investigate the $L^p$-boundedness of certain classes of periodic pseudo-differential operators. The operators considered arise from the study of symbols on $\mathbb{T}^n\times\mathbb{Z}^n$ with limited regularity.  
\end{abstract} \maketitle

\section{Introduction}

In this work we obtain $L^p-$boundedness theorems for pseudo-differential operators with symbols defined  on $\mathbb{T}^n\times\mathbb{Z}^n.$  For some recent work on boundedness results of periodic pseudo-differential operators in $L^p$-spaces we refer the reader to \cite{Duvan3,Profe, s2, Ruz, Ruz-2}. Pseudo-differential operators on $\mathbb{R}^n$ are  generalizations of differential operators and singular integrals. They are formally defined by
\begin{center}$T_{\sigma}u(x)=\int e^{i2\pi\langle x,\eta \rangle}\sigma(x,\eta)\widehat{u}(\eta)d\eta.$
\end{center}
The function $\sigma$ is usually called the symbol of the corresponding operator $T_{\sigma}.$ Symbols are classified according to their behavior and the behavior of their derivatives. For $m\in \mathbb{R}$ and  $0\leq \rho,\delta\leq 1,$ the $(\rho,\delta)-$H\"ormander class $S^{m}_{\rho,\delta}(\mathbb{R}^n\times \mathbb{R}^n)$ consists of those functions which are smooth in $(x,\eta)$ and which satisfy symbols inequalities
\begin{center}$|\partial^{\beta}_{x}\partial_{\eta}^{\alpha}\sigma(x,\eta)|\leq C_{\alpha,\beta}\langle \eta\rangle^{m-\rho|\alpha|+\delta|\beta|}.$
\end{center}
The corresponding set of operators with symbols in $(\rho,\delta)-$classes will be denoted by $\Psi^{m}_{\rho,\delta}(\mathbb{R}^n\times \mathbb{R}^n).$ A remarkable result due to A. P. Calder\'on and R. Vaillancourt, gives us that a pseudo-differential operator with symbol in $S^{0}_{\rho,\rho}(\mathbb{R}^n\times \mathbb{R}^n)$ for some $0\leq \rho<1$ is bounded on $L^2(\mathbb{R}^n),$ (see \cite{ca1, ca2}). Their result is false when $\rho=1:$ there exists symbols in $S^{0}_{1,1}(\mathbb{R}^n\times \mathbb{R}^n)$ whose associated pseudo-differential operators are  not bounded on $L^2(\mathbb{R}^n),$ (see  \cite{Duo}). For general $1<p<\infty$, we have the following theorem (see Fefferman \cite{fe}): if $m\leq -m_{p}= -n(1-\rho)|\frac{1}{p}-\frac{1}{2}|,$ $m_{p}<n(1-\rho)/2,$ then $T_{\sigma}\in \Psi^{m}_{\rho,\delta}(\mathbb{R}^n\times \mathbb{R}^n)$ is a $L^p$- bounded operator. It is well known that for every $m>m_p$ there exists $T_{\sigma}\in \Psi^{-m}_{\rho,\delta}(\mathbb{R}^n\times \mathbb{R}^n)$ which is not bounded on $L^{p}(\mathbb{R}^n).$  The historical development of the problem about the $L^p$-boundedness of pseudo-differential operators on $\mathbb{R}^n$ can be found in \cite{Na3,LW}.\\
\\
Pseudo-differential operators with symbols in H\"ormander classes can be defined on $C^{\infty}$- manifolds by using local charts. In 1979, Agranovich (see \cite{ag}) gives a global definition of pseudo-differential operators on the circle $\mathbb{S}^1,$ (instead of the local formulation on the circle as a manifold). By using the Fourier transform Agranovich's definition was readily generalizable to the $n-$dimensional torus $\mathbb{T}^n.$   It is a non-trivial result, that the definition of pseudo-differential operators with symbols on $(\rho,\delta)-$classes by Agranovich and H\"ormander are equivalent. This fact is known as the equivalence theorem of McLean (see \cite{Mc}).  Important consequences of this equivalence are the following periodic versions of the Calder\'on-Vaillancourt theorem and the Fefferman theorem:

\begin{theorem}\label{TFC}
Let $0\leq \delta<\rho\leq 1$ and let $a:\mathbb{T}^n\times \mathbb{Z}^n\rightarrow \mathbb{C}$ be a symbol with corresponding periodic pseudo-differential operator defined on $C^{\infty}(\mathbb{T}^n)$ by
\begin{equation}
a(x,D)f(x)=\sum_{\xi\in\mathbb{Z}^n } e^{ i2\pi \langle x,\xi \rangle}a(x,\xi)\widehat{f}(\xi).
\end{equation}
If $\alpha, \beta\in\mathbb{N}^n,$ under one of the following two conditions,
\begin{itemize}
\item[1.]$($Calderon-Vaillancourt condition, periodic version$).$ $$|\partial_{x}^{\beta}\Delta_{\xi}^{\alpha}a(x,\xi)  |\leq C_{\alpha,\beta} \langle \xi\rangle^{\rho(|\beta|-|\alpha|)},\,\,\,(x,\xi)\in\mathbb{T}^n\times \mathbb{Z}^n,$$
\item[2.]$($Fefferman condition, periodic version$).$ $$|\partial_{x}^{\beta}\Delta_{\xi}^{\alpha}a(x,\xi)  |\leq C_{\alpha,\beta} \langle \xi\rangle^{-m_{p}-\rho|\alpha|+\delta|\beta| },\,\,\,(x,\xi)\in\mathbb{T}^n\times \mathbb{Z}^n,\,\delta<1,$$  the operator $a(x,D)$ is bounded on $L^p(\mathbb{T}^n)$, for all  $1<p<\infty.$
\end{itemize}
\end{theorem}
\noindent Theorem \ref{TFC} is a consequence of the McLean equivalence theorem, the Calder\'on-Vaillancourt theorem and the  Fefferman theorem. Moreover, from the proofs of these results it is
easy to see that symbols with limited regularity satisfying conditions in Theorem \ref{TFC} give rise to $L^p$-bounded pseudo-differential operators (under the condition $0\leq \delta<\rho\leq 1$) (see \cite{va}).  With this in mind, in this paper we provide  some new results about the $L^p$ boundedness of periodic pseudo-differential operators associated to toroidal symbols with limited regularity. 
Now, we first recall some  recent $L^p$-theorems  of periodic operators and later we present our results. 

\begin{theorem}\cite[Wong-Molahajloo]{s2}. Let $a(x,D)$ be a periodic pseudo-differential operator on $\mathbb{T}^1\equiv\mathbb{S}^1 $. If  $\sigma:\mathbb{T}\times\mathbb{Z}\rightarrow \mathbb{C}$  satisfies 
\begin{equation}
|\partial_{x}^{\beta}\Delta_{\xi}^{\alpha}a(x,\xi)  |\leq C_{\alpha,\beta}\langle \xi\rangle^{-|\alpha|},\,\,\,(x,\xi)\in\mathbb{T}^n\times \mathbb{Z}^n,
\end{equation}
with $|\alpha|,|\beta|\leq 1,$ then, $a(x,D):L^p(\mathbb{T}^1)\rightarrow L^p(\mathbb{T}^1)$ is a bounded operator for all $1<p<\infty$.
\end{theorem}
 
\begin{theorem}\label{l2}\cite[Ruzhansky-Turunen]{Ruz} Let $n\in \mathbb{N},$ $k:=[\frac{n}{2}]+1$ and $a:\mathbb{T}^n\times \mathbb{Z}^n\rightarrow \mathbb{C}$ such that $|\partial^{\beta}_{x}a(x,\xi)|\leq C_{\beta},\,\,\, |\beta|\leq k.$ Then, $a(x,D):L^2(\mathbb{T}^n)\rightarrow L^2(\mathbb{T}^n)$ is a bounded operator.
\end{theorem} 

We note that compared with several well-known theorems on $L^2-$boundedness of pseudo-differential operators (Calderon-Vaillancourt theorem, for example), the Theorem \ref{l2} does not require any regularity with respect to the $\xi-$variable.
As a consequence of real interpolation and the $L^2-$estimate by Ruzhansky and Turunen, Delgado \cite{Profe} established the following sharp $L^p-$theorem. 

\begin{theorem}\label{p1}\cite[Delgado]{Profe} Let $0<\varepsilon <1$ and $k:=[\frac{n}{2}]+1,$  let $a:\mathbb{T}^n\times \mathbb{Z}^n\rightarrow \mathbb{C}$ be a symbol such that  $|\Delta_{\xi}^{\alpha}a(x,\xi)|\leq C_{\alpha}\langle \xi \rangle^{-\frac{n}{2}\varepsilon-(1-\varepsilon)|\alpha|},$ $|\partial_{x}^{\beta}a(x,\xi)|\leq C_{\beta}\langle \xi \rangle^{-\frac{n}{2}\varepsilon},$ for $|\alpha|,|\beta|\leq k.$ Then  $a(x,D):L^p(\mathbb{T}^n)\rightarrow L^p(\mathbb{T}^n)$ is a bounded linear operator, for $2\leq p<\infty.$
\end{theorem}
In the recent paper \cite{DR4} the authors have considered the $L^{p}$-boundedness of pseudo-differential operators on Compact Lie groups. Although, the original theorem is valid  for general compact Lie groups, we present the following periodic version.
\begin{theorem}\cite[Delgado-Ruzhansky]{DR4}\label{RuzDelg}
Let $0\leq \delta,\rho\leq 1$ and $1<p<\infty.$ Denote by $\kappa$ the smallest even integer larger than $\frac{n}{2}.$ Let $\sigma(x,D)$ be a pseudo-differential operator with symbol $\sigma$ satisfying
\begin{equation}\label{DRR}
|\partial_{x}^{\beta}\Delta_{\xi}^{\alpha}\sigma(x,\xi) |\leq C\langle \xi\rangle^{-m_0-\rho|\alpha|+\delta|\beta|},
\end{equation}
for all $|\alpha|\leq \kappa$ and $|\beta|\leq [\frac{n}{p}]+1,$ where $m_{0}:=\kappa (1-\rho) |1/p-1/2|+\delta([\frac{n}{p}]+1).$ Then $\sigma(x,D)$ extends to a bounded operator on $L^p(\mathbb{T}^n)$ for all $1<p<\infty.$
\end{theorem}

Now, we present the main results of this paper. Our starting point is the following result which consider the problem about the $L^{p}(\mathbb{T}^n)$-boundedness for  $1\leq p<\infty$, in contrast with the previous results where the case $p=1$ can be not considered because, as it is well known the boundedness of pseudo-differential operators with symbols in $(\rho,\delta)$- classes fails on $L^1.$ We denote $m_{p}:=-n(1-\rho)|\frac{1}{p}-\frac{1}{2}|,$ for all $1\leq p<\infty,$  $n\in\mathbb{N}$ and $0\leq \rho\leq 1.$\\
\\
\noindent \textbf{Theorem IA.}{ \it Let $0<\varepsilon\leq 1,$ $0\leq \rho\leq 1$  and let $\omega:[0,\infty)\rightarrow [0,\infty) $ be a non-decreasing and bounded function such that 
\begin{equation}
\int_{0}^1\omega(t)t^{-1}dt<\infty.
\end{equation}
Let  $\sigma(x,\xi):\mathbb{T}^n\times \mathbb{Z}^n\rightarrow\mathbb{C}$ be a symbol satisfying 
\begin{equation}
|\partial_{x}^{\beta}\Delta^{\alpha}_{\xi}\sigma(x,\xi)|\leq C\omega(\langle \xi \rangle^{-\varepsilon})\langle \xi \rangle^{-\frac{n}{2}(1-\rho)-\rho|\alpha|}
\end{equation}
for  $|\alpha|\leq [\frac{n}{2}]+1,$ $|\beta|\leq [\frac{n}{p}]+1,$  then the corresponding periodic operator $\sigma(x,D)$  extends to a bounded operator on $L^p(\mathbb{T}^n)$ for all $1\leq p<\infty.$}\\
\\
The following theorem is related with the result proved by Delgado and Ruzhansky mentioned above.  We discuss such relation in Remark \ref{mainremark}.\\
\\
\noindent  \textbf{Theorem IB.} { \it Let  $0< \rho\leq 1, $ and $1< p<\infty.$ Let $\sigma:\mathbb{T}^n\times \mathbb{Z}^n\rightarrow \mathbb{C}$ be a symbol satisfying
\begin{equation}\label{interpola}
|\partial_{x}^\beta\Delta^{\alpha}_{\xi}\sigma(x,\xi)|\leq C\langle \xi \rangle^{-m_p-\rho|\alpha|},
\end{equation}
for all $\alpha$ and $\beta$ in $\mathbb{N}^n$ such that  $|\alpha|\leq [\frac{n}{2}]+1,$ $|\beta|\leq [\frac{n}{p}]+1,$ then the corresponding periodic operator $\sigma(x,D)$ extends to  a bounded operator on $L^p(\mathbb{T}^n).$}
\vspace{0.2cm} \\

\begin{remark}We observe that from classical results (e.g.,   Lemma 2.2 of \cite{Na3}) together with suitable versions of the McLean equivalence theorem (see Proposition \ref{eqcc} below) the $L^p$-boundedness of a periodic operator satisfying \eqref{interpola} for $|\alpha|\leq [\frac{n}{2}]+1,$  $|\beta|\leq [\frac{n}{p}]+1,$ $2\leq p<\infty$ can be proved. For the case of $\mathbb{R}^n$ it is well known that symbols with $[\frac{n}{2}]+1$-derivatives in $\xi$ satisfying $(\rho,\delta)$-estimates does not always imply $L^p$-boundedness, $1<p\leq 2.$ However, in order to assure boundedness it is suffice to consider for this case $(n+1)$-derivatives in $\xi$. However, with our approach, we only need $k:=[\frac{n}{2}]+1$-derivatives in $\xi$ for a periodic symbol improving the immediate result that we could have  if we apply suitable versions of the McLean equivalence theorem (see Proposition \ref{eqcc}) (together with Lemma 2.2 of \cite{Na3}).  It is important to mention that the previous result can be not deduced from  Theorem \ref{p1}. In the following result we consider the boundedness of multipliers on $L^{\nu},$ $0<\nu\leq 1.$\end{remark}  

\noindent \textbf{Theorem IC.}  { \it Let $k:=[\frac{n}{2}]+1$. Let $\sigma(\xi)$ be a periodic symbol on $\mathbb{Z}^n$ satisfying
\begin{equation}
|\Delta^{\alpha}_{\xi}\sigma(\xi)|\leq C_{\alpha}\langle \xi \rangle^{-|\alpha|}, \,\,\,|\alpha|\leq k,
\end{equation}
then, the corresponding periodic operator $\sigma(D)$ is bounded from $L^1(\mathbb{T}^n)$ into $L^p(\mathbb{T}^n)$ for all $0<p<1.$}\vspace{0.2cm}

Interpolation via Riesz-Thorin Theorem  allow us to obtain the next $(L^p,L^r)$-theorem:\vspace{0.2cm}\vspace{0.2cm}

\noindent \textbf{Theorem II.}  { \it Let $2\leq p<\infty,$  $0<\varepsilon <1$ and $k:=[\frac{n}{2}]+1.$  Let $a:\mathbb{T}^n\times \mathbb{Z}^n\rightarrow \mathbb{C}$ be a symbol such that  $|\Delta_{\xi}^{\alpha}a(x,\xi)|\leq C_{\alpha}\langle \xi \rangle^{-\frac{n}{2}\varepsilon-(1-\varepsilon)|\alpha|},$ $|\partial_{x}^{\beta}a(x,\xi)|\leq C_{\beta}\langle \xi \rangle^{-\frac{n}{2}\varepsilon},$ for $|\alpha|,|\beta|\leq k$, then  $\sigma(x,D):L^p(\mathbb{T}^n)\rightarrow L^r(\mathbb{T}^n)$ is a bounded linear operator for all $1<q\leq r\leq p<\infty,$ where $\frac{1}{p}+\frac{1}{q}=1.$}\vspace{0.2cm}

An operator $T$ is positive if $f\geq 0$ implies that, the function $Tf$ is non-negative. In the following theorem we study positive and periodic  amplitude operators (see equation \eqref{amplitude} for the definition of amplitude operator):\vspace{0.2cm}

\noindent \textbf{Theorem III.}  { Let $0<\varepsilon,\delta <1.$  If $a(x,y,D)$ is a positive amplitude operator with symbol satisfying the following inequalities
\begin{equation}
|\partial_{x}^{\beta}\partial_{y}^{\alpha}a(x,y,\xi)|\leq C_{\alpha,\beta}\langle \xi\rangle^{\delta|\beta|-\varepsilon|\alpha|},\,\,\,|\alpha|,|\beta|\leq\mu:=[\frac{n}{2}(1-\delta)^{-1}]+1,
\end{equation}
then $a(x,y,D)$ is bounded on $L^1(\mathbb{T}^n).$ } \vspace{0.3cm}\\
\noindent Same as in Theorem \ref{l2}, symbols considered in Theorem III does not require any regularity condition on the Fourier variable. It is important to mention that there exists a connection between the $L^p$ boundedness of periodic operators and its continuity  on Besov spaces.  This relation has been studied by the author on general compact Lie groups in\cite[Section 3]{Duvan4}. Although some results in this paper consider the $L^p$-boundedness of pseudo-differential operators on the torus (for $1\leq p<\infty$), this problem has been addressed on general compact Lie groups in the references \cite{DR4,Fischer2} for all $1<p<\infty.$ Finally, we refer the reader to the references \cite{Ghaemi,Ghaemi2,Ghaemi3,m,Wong2} for other properties of periodic operators on $L^p$-spaces.    

\section{Preliminaries}

We use the standard notation of pseudo-differential operators (see \cite{f, Hor1, Hor2, Ruz, Wong}). The Schwartz space $\mathcal{S}(\mathbb{Z}^n)$ denote the space of functions $\phi:\mathbb{Z}^n\rightarrow \mathbb{C}$ such that 
 \begin{equation}
 \forall M\in\mathbb{R}, \exists C_{M}>0,\, |\phi(\xi)|\leq C_{M}\langle \xi \rangle^M,
 \end{equation}
where $\langle \xi \rangle=(1+|\xi|^2)^{\frac{1}{2}}.$ The toroidal Fourier transform is defined for any $f\in C^{\infty}(\mathbb{T}^n)$ by $\widehat{f}(\xi)=\int_{}e^{-i2\pi\langle x,\xi\rangle}f(x)dx,\,\,\xi\in\mathbb{Z}^n.$ The inversion formula is given by $f(x)=\sum_{}e^{i2\pi\langle x,\xi \rangle }\widehat{u}(\xi),\,\,x\in\mathbb{T}^n.$ The periodic H\"ormander class $S^m_{\rho,\delta}(\mathbb{T}^n\times \mathbb{R}^n), \,\, 0\leq \rho,\delta\leq 1,$ consists of those functions $a(x,\xi)$ which are smooth in $(x,\xi)\in \mathbb{T}^n\times \mathbb{R}^n$ and which satisfy toroidal symbols inequalities
\begin{equation}\label{css}
|\partial^{\beta}_{x}\partial^{\alpha}_{\xi}a(x,\xi)|\leq C_{\alpha,\beta}\langle \xi \rangle^{m-\rho|\alpha|+\delta|\beta|}.
\end{equation}
Symbols in $S^m_{\rho,\delta}(\mathbb{T}^n\times \mathbb{R}^n)$ are symbols in $S^m_{\rho,\delta}(\mathbb{R}^n\times \mathbb{R}^n)$ (see \cite{Hor1, Ruz}) of order $m$ which are 1-periodic in $x.$
If $a(x,\xi)\in S^{m}_{\rho,\delta}(\mathbb{T}^n\times \mathbb{R}^n),$ the corresponding pseudo-differential operator is defined by
\begin{equation}\label{hh}
a(x,D)u(x)=\int_{\mathbb{T}^n}\int_{\mathbb{R}^n}e^{i2\pi\langle x-y,\xi \rangle}a(x,\xi)u(y)d\xi dy,\,\, u\in C^{\infty}(\mathbb{T}^n).
\end{equation}
The set $S^m_{\rho,\delta}(\mathbb{T}^n\times \mathbb{Z}^n),\, 0\leq \rho,\delta\leq 1,$ consists of  those functions $a(x, \xi)$ which are smooth in $x$  for all $\xi\in\mathbb{Z}^n$ and which satisfy

\begin{equation}\label{cs}
\forall \alpha,\beta\in\mathbb{N}^n,\exists C_{\alpha,\beta}>0,\,\, |\Delta^{\alpha}_{\xi}\partial^{\beta}_{x}a(x,\xi)|\leq C_{\alpha,\beta}\langle \xi \rangle^{m-\rho|\alpha|+\delta|\beta|}.
 \end{equation}

The operator $\Delta_\xi^\alpha$ in  \eqref{cs} is the difference operator which is defined as follows. First, if $f:\mathbb{Z}^n\rightarrow \mathbb{C}$ is a discrete function and $(e_j)_{1\leq j\leq n}$ is the canonical basis of $\mathbb{R}^n,$
\begin{equation}
(\Delta_{\xi_{j}} f)(\xi)=f(\xi+e_{j})-f(\xi).
\end{equation}
If $k\in\mathbb{N},$ denote by $\Delta^k_{\xi_{j}}$  the composition of $\Delta_{\xi_{j}}$ with itself $k-$times. Finally, if $\alpha\in\mathbb{N}^n,$ $\Delta^{\alpha}_{\xi}= \Delta^{\alpha_1}_{\xi_{1}}\cdots \Delta^{\alpha_n}_{\xi_{n}}.$
The toroidal operator (or periodic operator) with symbol $a(x,\xi)$ is defined as
\begin{equation}\label{aa}
a(x,D)u(x)=\sum_{\xi\in\mathbb{Z}^n}e^{i 2\pi\langle x,\xi\rangle}a(x,\xi)\widehat{u}(\xi),\,\, u\in C^{\infty}(\mathbb{T}^n).
\end{equation}

There exists a process to interpolate the second argument of symbols on $\mathbb{T}^n\times \mathbb{Z}^n$ in a smooth way to get a symbol defined on $\mathbb{T}^n\times \mathbb{R}^n.$
\begin{proposition}\label{eq}
Let $0\leq \delta \leq 1,$ $0< \rho\leq 1.$ The symbol $a\in S^m_{\rho,\delta}(\mathbb{T}^n\times \mathbb{Z}^n)$ if only if there exists  a Euclidean symbol $a'\in S^m_{\rho,\delta}(\mathbb{T}^n\times \mathbb{R}^n)$ such that $a=a'|_{\mathbb{T}^n\times \mathbb{Z}^n}.$
\end{proposition}
\begin{proof} The proof can be found in \cite{Mc, Ruz}.
\end{proof}

It is a non trivial fact, however, that the definition of pseudo-differential operator on a torus given by Agranovich (equation \ref{aa} )  and H\"ormander (equation \ref{hh}) are equivalent. McLean (see \cite{Mc}) prove this for all the H\"ormander classes $S^m_{\rho,\delta}(\mathbb{T}^n\times \mathbb{Z}^n).$ A different proof to this fact can be found in \cite{Ruz}, Corollary 4.6.13.

\begin{proposition}$\label{eqc}($Equality of Operators Classes$).$ For $0\leq \delta \leq 1,$ $0<\rho\leq 1$ we have $\Psi^{m}_{\rho,\delta}(\mathbb{T}^n\times \mathbb{Z}^n)=\Psi^{m}_{\rho,\delta}(\mathbb{T}^n\times \mathbb{R}^n).$
\end{proposition}
A look at the proof (based in Theorem 4.5.3 of \cite{Ruz}) of the Proposition \ref{eqc} shows us that a more general version is still valid for symbols with limited regularity as follows (see Corollary 4.5.7 of \cite{Ruz}):

\begin{corollary}\label{eqc'} Let $0\leq \delta \leq 1,$ $0\leq \rho<1.$ Let $a:\mathbb{T}^n\times \mathbb{R}^n\rightarrow \mathbb{C}$ satisfying $(\rho,\delta)-$inequalities for $|\alpha|\leq N_{1}$ and $|\beta|\leq N_{2}.$ Then the restriction $\tilde{a}=a|_{\mathbb{T}^n\times \mathbb{Z}^n}$ satisfies $(\rho,\delta)-$estimates for $|\alpha|\leq N_{1}$ and $|\beta|\leq N_{2}.$  The converse holds true, i.e, if every symbol on $\mathbb{T}^n\times \mathbb{Z}^n$ satisfying $(\rho,\delta)$-inequalities (as in \eqref{cs}) is the restriction of a symbol on $\mathbb{T}^n\times \mathbb{R}^n$ satisfying $(\rho,\delta)$ inequalities as in \eqref{css}. 
\end{corollary}
Let us denote $\Psi^{m}_{\rho,\delta,N_1,N_2}(\mathbb{T}^n\times \mathbb{Z}^n)$ to the set of operators associated to symbols satisfying \eqref{cs} for all $|\alpha|\leq N_{1}$ and $|\beta|\leq N_2,$ and $\Psi^{m}_{\rho,\delta,N_1,N_2}(\mathbb{T}^n\times \mathbb{R}^n)$ defined similarly. Then we have (see Theorem 2.14 of \cite{Profe}):
\begin{proposition}
$\label{eqcc}($Equality of Operators Classes$).$ For $0\leq \delta \leq 1,$ $0<\rho\leq 1$ we have $\Psi^{m}_{\rho,\delta,N_1,N_2}(\mathbb{T}^n\times \mathbb{Z}^n)=\Psi^{m}_{\rho,\delta,N_1,N_2}(\mathbb{T}^n\times \mathbb{R}^n).$
\end{proposition}
The toroidal calculus is closed under adjoint operators. The corresponding announcement is the following. 
\begin{proposition}\label{ad} Let $0\leq \delta <\rho\leq 1.$ Let $a(x,D)$ be a operator with symbol $a(X,\xi)\in S^m_{\rho,\delta}(\mathbb{T}^n\times \mathbb{Z}^n).$ Then, the adjoint $a^{*}(x,D)$ of $a(x,D),$ has symbol $a^{*}(x,\xi)\in S^m_{\rho,\delta}(\mathbb{T}^n\times \mathbb{Z}^n).$ The symbol $a^{*}(x,\xi)$ has the following asymptotic expansion: 
\begin{equation} \sigma^{*}(x,\xi)\approx \sum_{\alpha \geq 0}\frac{1}{\alpha !}\Delta^{\alpha}_{\xi}D^{(\alpha)}_{x}\overline{\sigma(x,\xi)}.
\end{equation}
Moreover, if $n_{0}\in \mathbb{N},$ then 
\begin{equation} \sigma^{*}(x,\xi)- \sum_{|\alpha|<n_{0} }\frac{1}{\alpha !}\Delta^{\alpha}_{\xi}D^{(\alpha)}_{x}\overline{\sigma(x,\xi)}\,\,\,  \in S^{m-n_{0}(\rho-\delta)}_{\rho,\delta}(\mathbb{T}^n\times \mathbb{Z}^n).
\end{equation}
\end{proposition}
In order to establish our  result on positive operators,  we introduce amplitude operators. The periodic amplitudes  are functions $a(x,y,\xi)$ defined on $\mathbb{T}^n\times\mathbb{T}^n\times\mathbb{Z}^n.$ The corresponding amplitude operators are defined as
\begin{equation}\label{amplitude}a(x,y,D)u(x)=\int\limits_{\mathbb{T}^n\times \mathbb{R}^n} e^{i2\pi \langle x-y,\xi \rangle}a(x,y,\xi)u(y)dy\,d\xi.
\end{equation}
If the symbol depends only on $(x,\xi)-$variables then $p(x,y,D)=p(x,D).$ Moreover, if a symbol $\sigma(x,\xi)=\sigma(\xi)$ depends only on the Fourier variable $\xi,$ the corresponding pseudo-differential operator $\sigma(x,D)=\sigma(D)$ is called a Fourier multiplier. An instrumental result on Fourier multipliers in the proof of our main results is the following: (see, Theorem 3.8 of Stein  \cite{Eli}).

\begin{proposition}\label{stein}
Suppose $1\leq p\leq \infty$ and $T_\sigma$ be a Fourier multiplier on $ \mathbb{R}^n$ with symbol $\sigma(\xi)$. If  $\sigma(\xi)$  is continuous at each point of $\mathbb{Z}^n$ then the periodic operator defined by
\begin{equation}
\sigma(D)f(x)=\sum_{\xi \in \mathbb{Z}^n}e^{i2\pi \langle x,\xi\rangle}\sigma(\xi)\widehat{u}(\xi),
\end{equation}
is a bounded operator from $L^p(\mathbb{T}^n)$ into $L^p(\mathbb{T}^n).$
\end{proposition}
The following results will help clarify the nature of the conditions that will be imposed on periodic symbols in order to obtain $L^p$ theorems for periodic operators with symbols of limited regularity.

\begin{proposition}\label{remi1}
Let $0\leq \rho\leq 1$ and $0<\varepsilon \leq 1,$ and suppose that the symbol $\sigma(x,\xi)$ on $\mathbb{R}^n\times \mathbb{R}^n$ satisfies
\begin{equation}
|\partial_{\xi}^{\alpha}\sigma(x,\xi)|\leq C\omega(\langle \xi \rangle^{-\varepsilon} )\langle \xi \rangle^{-\frac{n}{2}(1-\rho)-\rho|\alpha|},\,\,\,\,|\alpha|\leq [\frac{n}{2}]+1,
\end{equation}
where $\omega$ is a non-decreasing, bounded and non-negative function on $[0,\infty)$ satisfying $$\int_{0}^{1}\omega(t)t^{-1}dt <\infty.$$ Then $T_{\sigma}$ is a bounded operator on $L^p(\mathbb{R}^n)$ for all $2\leq p\leq \infty.$
\end{proposition}
\begin{proof}
See Theorem 4.4 and Corollary 4.2 of \cite{va}.
\end{proof}
The following is a version of the Fefferman theorem but symbols are considered with limited smoothness. 
\begin{proposition}\label{remi2}
Let $2\leq p<\infty$ and $0\leq \delta\leq \rho\leq 1,$ $\delta<1.$ Let $\sigma(x,\xi)$ be a symbol satisfying
\begin{equation}
|\partial_{x}^{\beta}\partial_{\xi}^{\alpha}\sigma(x,\xi)|\leq C\langle \xi \rangle^{-m_{p}-\rho|\alpha|+\delta|\beta|},\,\,\,\,|\alpha|,|\beta|\leq [\frac{n}{2}]+1,
\end{equation}
where $m_{p}=n(1-\rho)|\frac{1}{p}-\frac{1}{2}|.$  Then $T_\sigma:L^p(\mathbb{R}^n)\rightarrow L^p(\mathbb{R}^n)$ is a bounded operator.
\end{proposition}
\begin{proof} See Theorem 5.1 and Corollary 5.1 of \cite{va}.
\end{proof}
The following theorem is the particular case of one proved in \cite{Ruz3} for compact Lie groups.
\begin{theorem}\label{ruz3}
Let $k>\frac{n}{2}$ be an even integer. If $\sigma(\xi)$ is a periodic symbol on $\mathbb{Z}^n$ satisfying
\begin{equation}
|\Delta^{\alpha}_{\xi}\sigma(\xi)|\leq C_{\alpha}\langle \xi \rangle^{-|\alpha|},  \,\,\,|\alpha|\leq k,
\end{equation}
then, the corresponding periodic operator $\sigma(D)$ is of weak type (1,1) and $L^p$-bounded for all $1<p<\infty.$
\end{theorem}
\noindent Also, weak(1,1) boundedness of periodic operators has been considered by the author in \cite{Duvan3}.
We end this section with the following result proved in \cite{DR4}. Although, the original theorem is valid  for general compact Lie groups, we present the periodic version for simplicity.
\begin{theorem}\label{Ruz-Delg}
Let $0\leq \delta,\rho\leq 1$ and $1<p<\infty.$ Denote by $\kappa$ the smallest even integer larger than $\frac{n}{2}.$ Let $\sigma(x,D)$ be a pseudo-differential operator with symbol $\sigma$ satisfying
\begin{equation}\label{DRR}
|\partial_{x}^{\beta}\Delta_{\xi}^{\alpha}\sigma(x,\xi) |\leq C\langle \xi\rangle^{-m_0-\rho|\alpha|+\delta|\beta|},
\end{equation}
for all $|\alpha|\leq \kappa$ and $|\beta|\leq [\frac{n}{p}]+1,$ where $m_{0}:=\kappa (1-\rho) |1/p-1/2|+\delta([\frac{n}{p}]+1).$ Then $\sigma(x,D)$ extends to a bounded operator on $L^p(\mathbb{T}^n)$ for all $1<p<\infty.$
\end{theorem}

\section{Main Results-Proofs}

In this section we prove our main results. we discuss that conditions on the periodic symbol $\sigma(x,\xi)$ guarantee the $L^p-$boundedness of the corresponding pseudo-differential operator.\\ 

\begin{lemma}\label{LE1}
Let $0<\varepsilon\leq 1,$ $0\leq \rho\leq 1, $ $1\leq p\leq\infty$ and $\omega:[0,\infty)\rightarrow [0,\infty) $ be a non-decreasing function such that 
\begin{equation}
\int_{0}^1\omega(t)t^{-1}dt<\infty.
\end{equation}
If $\sigma_{1}: \mathbb{R}^n\rightarrow \mathbb{C}$ is a symbol satisfying
\begin{equation}
|\partial^{\alpha}_{\xi}\sigma_{1}(\xi)|\leq C\omega(\langle \xi \rangle^{-\varepsilon})\langle \xi \rangle^{-\frac{n}{2}(1-\rho)-\rho|\alpha|}
\end{equation}
for all $\alpha$  with $|\alpha|\leq [\frac{n}{2}]+1,$  then the corresponding periodic operator $\sigma(D)$ with symbol $\sigma(\xi)=\sigma_{1}(\xi)|_{ \mathbb{Z}^n}$ is a bounded operator on $L^p(\mathbb{T}^n).$
\end{lemma}
\begin{proof} Proposition \ref{remi1} provides the $L^p(\mathbb{R}^n)$-boundedness of $T_{\sigma_1}.$ Now, by Proposition \ref{stein}, the pseudo-differential operator with symbol $\sigma(\xi)$ is $L^{p}(\mathbb{T}^n)-$bounded.
\end{proof}

\begin{theorem}\label{IA} Let $0<\varepsilon\leq 1,$ $0\leq \rho\leq 1, $ $1\leq p<\infty$ and $\omega:[0,\infty)\rightarrow [0,\infty) $ be a bounded and non-decreasing function such that 
\begin{equation}
\int_{0}^1\omega(t)t^{-1}dt<\infty.
\end{equation}
If $\sigma:\mathbb{T}^n\times \mathbb{R}^n\rightarrow \mathbb{C}$ is a symbol satisfying
\begin{equation}
|\partial^{\alpha}_{\xi} \partial_{x}^\beta\sigma(x,\xi)|\leq C\omega(\langle \xi \rangle^{-\varepsilon})\langle \xi \rangle^{-\frac{n}{2}(1-\rho)-\rho|\alpha|}
\end{equation}
for $|\alpha|\leq [\frac{n}{2}]+1,$ $|\beta|\leq [\frac{n}{p}]+1,$ then the corresponding periodic operator $\sigma(x,D)$  is a bounded linear operator on $L^p(\mathbb{T}^n).$
\end{theorem}
\begin{proof} 
Let us consider the Schwartz kernel $K(x,z)$ of $\sigma(x,D)$ which is given by $K(x,y)=r(x-y,x)$ where
$
r(x,z)=\int_{\mathbb{R}^n}e^{i2\pi \langle x,\xi \rangle }\sigma(z,\xi),
$
is understood in the distributional sense.
For every $z\in\mathbb{T}^n$ fixed, $r(z)(\cdot)=r(\cdot,z)$ is a distribution and the map $f\mapsto f\ast r(\cdot,z)=f\ast r(z)(\cdot)$ is a pseudo-differential operator with symbol $\sigma_{z}:\xi\mapsto \sigma(z,\xi).$ If $x\in\mathbb{T}^n$ and $f\in C^{\infty}(\mathbb{T}^n),$ $\sigma(x,D)f(x)=(f\ast r(x))(x).$ Moreover, for any $\beta\in \mathbb{N}^n$ with $|\beta|\leq [\frac{n}{p}]+1,$ the pseudo-differential operator  $f\ast [\partial^{\beta}_{x}r(\cdot,x)]|_{x=z}:=f\ast [\partial^{\beta}_{z}r(\cdot,z)]$ is a pseudo-differential operator with symbol $[\partial^{\beta}_{x}\sigma(\cdot,x)]|_{x=z}:=\partial^{\beta}_{z}\sigma(z,\cdot).$ By Lemma \ref{LE1}, every pseudo-differential operator $\sigma_{z,\beta}(D)$ with symbol $\sigma_{z,\beta}(\xi)=\partial^{\beta}_{z}\sigma(z,\xi)$ is $L^p(\mathbb{T}^n)$-bounded for $1\leq p< \infty$. Now, by the Sobolev embedding theorem, for $1\leq p<\infty,$ 
\begin{align*}
|\sigma_(x,D)f(x)|^p &\leq \sup_{z\in\mathbb{T}^n} |(f\ast r(z))(x)|^p
\leq \sum_{|\beta|\leq [\frac{n}{p}]+1}\int_{\mathbb{T}^n}|  (f\ast \partial^{\beta}_{z}r(z))(x) |^pdz
\end{align*}
Hence, by application of the Fubini theorem we get
\begin{align*}
\Vert \sigma(x,D)f \Vert^{p}_{L^p(\mathbb{T}^n)} & \leq C^p \sum_{|\beta|\leq [\frac{n}{p}]+1}\int_{\mathbb{T}^n}\int_{\mathbb{T}^n}|  (f\ast \partial^{\beta}_{z}r(z))(x) |^pdzdx\\
 & \leq C^p\sum_{|\beta|\leq [\frac{n}{p}]+1}\int_{\mathbb{T}^n}\int_{\mathbb{T}^n}|  (f\ast \partial^{\beta}_{z}r(z))(x) |^pdxdz\\
 &\leq C^p \sum_{|\beta|\leq [\frac{n}{p}]+1 }\sup_{z\in\mathbb{T}^n} \int_{\mathbb{T}^n}|  (f\ast \partial^{\beta}_{z}r(z))(x) |^pdx\\&= C^p\sum_{|\beta|\leq [\frac{n}{p}]+1 }\sup_{z\in\mathbb{T}^n}\Vert \sigma_{z,\beta}(D)f \Vert^p_{L^p(\mathbb{T}^n)} \\
 &\leq C^p\left(  \sum_{|\beta|\leq [\frac{n}{p}]+1 }\sup_{z\in\mathbb{T}^n}\Vert \sigma_{z,\beta}(D) \Vert^p_{B(L^p,L^p)}    \right)\Vert f\Vert^p_{L^p(\mathbb{T}^n)}.
\end{align*}
Thus,
$$ \Vert \sigma(x,D)f \Vert_{L^p(\mathbb{T}^n)}\leq C \left(  \sum_{|\beta|\leq [\frac{n}{p}]+1 }\sup_{z\in\mathbb{T}^n}\Vert \sigma_{z,\beta}(D) \Vert^p_{B(L^p,L^p)}    \right)^{\frac{1}{p}}\Vert f \Vert_{L^p(\mathbb{T}^n)}.  $$
\end{proof}
\begin{lemma}\label{lemmaIA}
Let $k\in\mathbb{R},$ $\varepsilon>0,$ $\omega$ be a function as in Proposition \ref{remi1} and $\sigma:\mathbb{T}^n\times \mathbb{R}^n\rightarrow \mathbb{C}$  be a symbol satisfying
\begin{equation}\label{w1}
|\partial^{\alpha}_{\xi} \partial_{x}^\beta\sigma(x,\xi)|\leq C\omega(\langle \xi \rangle^{-\varepsilon})\langle \xi \rangle^{k},
\end{equation}
for all $|\alpha|\leq N_1$ and $|\beta|\leq N_2.$ Let $\tilde{a}(x,\xi):=\sigma(x,\xi)|_{\mathbb{T}^{n}\times \mathbb{Z}^n}.$ Then
\begin{equation}\label{w2}
|\Delta^{\alpha}_{\xi} \partial_{x}^\beta\tilde{\sigma}(x,\xi)|\leq C'\omega(\langle \xi \rangle^{-\varepsilon})\langle \xi \rangle^{k},
\end{equation}
for all $|\alpha|\leq N_1$ and $|\beta|\leq N_2.$ Moreover, every symbol satisfying  \eqref{w2} is the restriction of a symbol on $\mathbb{T}^n\times \mathbb{R}^n$ satisfying \eqref{w1}.
\end{lemma}
\begin{proof}
Let us consider $\sigma$ as in \eqref{w1}. By the mean value theorem, if $|\alpha|=1$ we have
\begin{align*}
\Delta_{\xi}^{\alpha}\partial^{\beta}_{x}\tilde{\sigma}(x,\xi)=\Delta_{\xi}^{\alpha}\partial^{\beta}_{x}{\sigma}(x,\xi)=\partial_{\xi}^{\alpha}\partial^{\beta}_{x}{\sigma}(x,\xi)|_{\xi=\eta}
\end{align*}
where $\eta$ is on the line $[\xi,\xi+\alpha].$ For a general multi-index $\alpha\in\mathbb{N}^n,$ it can proved by induction that
\begin{align*}
\Delta_{\xi}^{\alpha}\partial^{\beta}_{x}\tilde{\sigma}(x,\xi)=\partial_{\xi}^{\alpha}\partial^{\beta}_{x}{\sigma}(x,\xi)|_{\xi=\eta}
\end{align*}
for some $\eta\in Q:=[\xi_{1}\times \xi_{1}+\alpha_1]\times \cdots [\xi_{n}\times \xi_{n}+\alpha_n]. $ Hence, we have
\begin{align*}
|\Delta_{\xi}^{\alpha}\partial_{x}^{\beta}\tilde{\sigma}(x,\xi)| &=|\partial_{\xi}^{\alpha}\partial^{\beta}_{x}{\sigma}(x,\xi)|_{\xi =\eta\in Q}|\\
&\leq C\omega(\langle \eta \rangle^{-\varepsilon})\langle \eta \rangle^{k}\leq C'\omega(\langle \xi \rangle^{-\varepsilon})\langle \xi \rangle^{k}.
\end{align*}
So, we have proved the first part of the theorem. Now, let us consider a symbol $\tilde{\sigma}$ on $\mathbb{T}^n\times \mathbb{Z}^n$ satisfying \eqref{w2}. Let us consider $\theta$ as in Lemma 4.5.1 of \cite{Ruz}. Define the symbol $\sigma$ on $\mathbb{T}^n\times \mathbb{R}^n$ by
\begin{equation}
\sigma(x,\xi)=\sum_{\eta\in\mathbb{Z}^n}(\mathcal{F}_{\mathbb{R}^n}\theta)(\xi-\eta)\tilde{\sigma}(x,\eta).
\end{equation}
Same as in the proof of Theorem 4.5.3 of \cite{Ruz}, pag. 359,  we have
\begin{align*}
|\partial_{x}^{\beta}\partial_{\xi}^{\alpha}\sigma(x,\xi)|=\left| \sum_{\eta\in\mathbb{Z}^n}\phi_{\alpha}(\xi-\eta)\partial_{x}^{\beta}\Delta_{\xi}^{\alpha}\tilde{\sigma}(x,\eta) \right|
\end{align*}
where every $\phi_{\alpha}\in\mathcal{S}(\mathbb{R}^n)$ is a function as in Lemma 4.5.1 of \cite{Ruz}. So, we obtain 
\begin{equation}
|\partial_{x}^{\beta}\partial_{\xi}^{\alpha}\sigma(x,\xi)|\leq \sum_{\eta\in\mathbb{Z}^n}|\phi_{\alpha}(\eta)\omega(\langle \xi-\eta \rangle^{-\varepsilon})\langle \xi-\eta\rangle^{k}|
\end{equation}
Since $\omega$ is bounded, for some $M>0,$ and all $\alpha\geq 1$ we have
\begin{equation}\label{omega}|\omega(\alpha t)|\leq \alpha\omega(t),\,\,t>M.\end{equation}
By \eqref{omega}, the Peetre inequality and using that $\omega$ is increasing we have
\begin{align*}
|\partial_{x}^{\beta}\partial_{\xi}^{\alpha}\sigma(x,\xi)| &\leq \sum_{\eta\in\mathbb{Z}^n}|\phi_{\alpha}(\eta)\omega(\langle \xi-\eta \rangle^{-\varepsilon})\langle \xi-\eta\rangle^{k} |\\
&\leq \sum_{\eta\in\mathbb{Z}^n}|\phi_{\alpha}(\eta)\omega(\langle \xi\rangle^{-\varepsilon}\langle\eta \rangle^{\varepsilon})\langle \xi\rangle^{k}\langle \eta\rangle^{k} |\\
&\leq C \omega(\langle \xi\rangle^{-\varepsilon})\langle \xi\rangle^{k}\sum_{\eta\in\mathbb{Z}^n}|\phi_{\alpha}(\eta)\langle\eta \rangle^{\varepsilon+k}|.
\end{align*}
Since every $\phi_{\alpha}$ is a function in the Schwartz class we obtain
\begin{equation}
|\partial_{x}^{\beta}\partial_{\xi}^{\alpha}\sigma(x,\xi)| \leq C'\omega(\langle \xi\rangle^{-\varepsilon})\langle \xi\rangle^{k}.
\end{equation}
So, we end the proof.
\end{proof}
As a consequence of the results above we obtain the following result.\\
\\
\noindent \textbf{Theorem IA.}{ \it Let $0<\varepsilon\leq 1,$ $0\leq \rho\leq 1$  and let $\omega:[0,\infty)\rightarrow [0,\infty) $ be a non-decreasing and bounded function such that 
\begin{equation}
\int_{0}^1\omega(t)t^{-1}dt<\infty.
\end{equation}
Let  $\sigma(x,\xi):\mathbb{T}^n\times \mathbb{Z}^n\rightarrow\mathbb{C}$ be a symbol satisfying 
\begin{equation}
|\partial_{x}^{\beta}\Delta^{\alpha}_{\xi}\sigma(x,\xi)|\leq C\omega(\langle \xi \rangle^{-\varepsilon})\langle \xi \rangle^{-\frac{n}{2}(1-\rho)-\rho|\alpha|}
\end{equation}
for  $\alpha,\beta\in\mathbb{N}^n$ such that $|\alpha|\leq [\frac{n}{2}]+1,$ $|\beta|\leq [\frac{n}{p}]+1,$  then the corresponding periodic operator $\sigma(x,D)$  extends to a bounded operator on $L^p(\mathbb{T}^n)$ for all $1\leq p<\infty.$}\\
\\
\begin{proof}
First, let us denote by $\Psi^{m,\omega}_{\rho,\delta, N_1,N_2}(\mathbb{T}^n\times \mathbb{R}^n)$ and  $\Psi^{m,\omega}_{\rho,\delta, N_1,N_2}(\mathbb{T}^n\times \mathbb{Z}^n)$ to the set of operators with symbols satisfying \eqref{w1} and \eqref{w2} respectively. If we combine the Theorem 4.6.12 of \cite{Ruz} and Lemma \ref{lemmaIA} we obtain the equality of classes
\begin{equation}
\Psi^{m,\omega}_{\rho,\delta, N_1,N_2}(\mathbb{T}^n\times \mathbb{R}^n)=\Psi^{m,\omega}_{\rho,\delta, N_1,N_2}(\mathbb{T}^n\times \mathbb{Z}^n).
\end{equation}
Hence, in order to proof the boundedness of $\sigma(x,D)$ we only need to proof that 
\begin{equation}
\Psi^{-\frac{n}{2}(1-\rho),\,\omega}_{\rho,0, [\frac{n}{2}]+1,[\frac{n}{p}]+1}(\mathbb{T}^n\times \mathbb{R}^n)\subset\mathcal{L}(L^{p}(\mathbb{T}^n)),
\end{equation}
but, this fact has been proved in Theorem \ref{IA}.
\end{proof}

\begin{lemma}\label{LE2} Let  $0\leq \rho\leq 1, $ and $1< p<\infty.$ Let $\sigma_{1}: \mathbb{R}^n\rightarrow \mathbb{C}$ be a symbol satisfying
\begin{equation}
|\partial^{\alpha}_{\xi}\sigma_{1}(\xi)|\leq C\langle \xi \rangle^{-m_p-\rho|\alpha|}
\end{equation}
for all $\alpha$ and $\beta$ with $|\alpha|\leq [\frac{n}{2}]+1,$  then the corresponding periodic operator $\sigma(D)$ with symbol $\sigma(\xi)=\sigma_{1}(\xi)|_{ \mathbb{Z}^n}$ is a bounded operator on $L^p(\mathbb{T}^n).$
\end{lemma}
\begin{proof}
From Proposition \ref{remi2}, we deduce the $L^p(\mathbb{R}^n)-$boundedness of $T_{\sigma_1}.$ Finally, by Proposition  \ref{stein}, $\sigma(D)$ is a bounded operator on $L^{p}(\mathbb{T}^n).$
\end{proof}

\noindent \textbf{Theorem IB.}  { \it  Let  $0< \rho\leq 1, $ and $1< p<\infty.$ Let $\sigma:\mathbb{T}^n\times \mathbb{Z}^n\rightarrow \mathbb{C}$ be a symbol satisfying
\begin{equation}\label{cardona}
|\partial_{x}^{\beta}\Delta_{\xi}^{\alpha}\sigma(x,\xi) |\leq C\langle \xi\rangle^{-m_p-\rho|\alpha|},
\end{equation}
for all $\alpha$ and $\beta$ with $|\alpha|\leq [\frac{n}{2}]+1,$ $|\beta|\leq [\frac{n}{p}]+1,$ then  $\sigma(x,D)$ extends to a bounded operator on $L^p(\mathbb{T}^n).$}
\vspace{0.2cm}\\
\begin{proof} By Proposition \ref{eqcc} we only need to proof that 
\begin{equation}
\Psi^{-m_p}_{\rho,0,[\frac{n}{2}]+1,[\frac{n}{p}]+1}(\mathbb{T}^n\times \mathbb{R}^n)\subset \mathcal{L}(L^{p}(\mathbb{T}^n)),\,\,1<p<\infty.
\end{equation}
Hence, let us consider a symbol $\sigma_{1}$ on $\mathbb{T}^n\times \mathbb{R}^n$ satisfying \begin{equation}\label{interpolac}
|\partial_{x}^\beta\partial^{\alpha}_{\xi}\sigma_{1}(x,\xi)|\leq C\langle \xi \rangle^{-m_p-\rho|\alpha|},
\end{equation}for all  $|\alpha|\leq [\frac{n}{2}]+1,$ $|\beta|\leq [\frac{n}{p}]+1.$ 
From Lemma \ref{LE2}, we obtain the $L^{p}-$boundedness of every operator ${\sigma_1}_{z,\beta}(D)$ with symbol ${\sigma_1}_{z,\beta}(\xi)=(\partial_{x}^{\beta}\sigma_{1}(x,\xi))|_{x=z},$ $z\in\mathbb{T}^n.$ As in the proof of Theorem IA, an application of the Sobolev embedding theorem gives
$$ \Vert \sigma_{1}(x,D)f \Vert_{L^p(\mathbb{T}^n)}\leq C \left(  \sum_{|\beta|\leq [\frac{n}{p}]+1 }\sup_{z\in\mathbb{T}^n}\Vert {\sigma_1}_{z,\beta}(D) \Vert^p_{B(L^p,L^p)}    \right)^{\frac{1}{p}}\Vert f \Vert_{L^p(\mathbb{T}^n)}.$$

Hence $\sigma_1(x,D)$ is $L^p$-bounded. So we end the proof.
\end{proof}
\begin{remark}\label{mainremark}
We observe that conditions on the number of derivatives in Theorem \ref{Ruz-Delg} and Theorem IB agree, but the  behavior of the symbol derivatives in every theorem is different. In fact, \eqref{cardona} can be written as 
\begin{equation}
|\partial_{x}^{\beta}\Delta_{\xi}^{\alpha}\sigma(x,\xi) |\leq C\langle \xi\rangle^{-n(1-\rho)|1/p-1/2|-\rho|\alpha|},\,\,\,\,\,\,\,\,\,\,\,\,\,\,\,\,\,\,\,\,\,\,\,\,\,\,\,\,\,\,\,
\end{equation}
in  contrast with \eqref{DRR}:
\begin{equation}
|\partial_{x}^{\beta}\Delta_{\xi}^{\alpha}\sigma(x,\xi) |\leq C\langle \xi\rangle^{-\kappa(1-\rho)|1/p-1/2|-\rho|\alpha|+\delta(|\beta|-([\frac{n}{p}]+1))},
\end{equation}
where $|\beta|-([\frac{n}{p}]+1)\leq 0.$
\end{remark}
\begin{lemma}\label{Kol1} $($Kolmogorov's lemma$).$ Given an operator $S$ which is weak(1,1), $0<v<1,$ and a set $E$ of finite measure, there exists a $C>0$ such that
\begin{equation}\int_{E}|Sf(x)|^v\,dx\leq C\mu (E)^{1-v}\Vert f\Vert^v_{L^1(\mathbb{R}^n)}. 
\end{equation}
\end{lemma}
\begin{proof}
For the proof of this result, see Lemma 5.16 of \cite{Duo}.
\end{proof}

\noindent \textbf{Theorem IC.}  { \it  Let $k:=[\frac{n}{2}]+1$. If $\sigma(\xi)$ is a periodic symbol on $\mathbb{Z}^n$ satisfying
\begin{equation}
|\Delta^{\alpha}_{\xi}\sigma(\xi)|\leq C_{\alpha}\langle \xi \rangle^{-|\alpha|}, \,\,\,|\alpha|\leq k, 
\end{equation}
then, the corresponding periodic operator $\sigma(D)$ is bounded from $L^1(\mathbb{T}^n)$ into $L^p(\mathbb{T}^n)$ for all $0<p<1.$}\vspace{0.2cm}
\begin{proof} By Theorem \ref{eqcc} we only need to prove that an operator $\sigma_{1}(D)$ with symbol $\sigma_{1}$ on $\mathbb{T}^n\times \mathbb{R}^n$ satisfying
\begin{equation}
|\partial^{\alpha}_{\xi}\sigma_1(\xi)|\leq C_{\alpha}\langle \xi \rangle^{-|\alpha|}, \,\,\,|\alpha|\leq k, 
\end{equation}
is a bounded operator from $L^1(\mathbb{T}^n)$ into $L^p(\mathbb{T}^n)$ for all $0<p<1.$
The multiplier $\sigma(D)$ is of weak type (1,1) on $\mathbb{R}^n$ (see \cite{Ste,Eli}). Now, by Kolmogorov's lemma with $E=\mathbb{T}^n,$  for $f\in L^{p}(\mathbb{T}^n)$ (which can be considered as a function on $\mathbb{R}^n$ equal to zero in $\mathbb{R}^n-\mathbb{T}^n$) we have
$$\int_{\mathbb{T}^n}|\sigma_{1}(x,D)f(x)|^p\,dx\leq C\Vert f\Vert^p_{L^1(\mathbb{R}^n)}=C\Vert f\Vert^p_{L^1(\mathbb{T}^n)}$$
which proves the boundedness of $\sigma_{1}(x,D).$ So, we end the proof.
\end{proof}

\begin{lemma}\label{lemma22} Let $2\leq p<\infty,$  and $k=[\frac{n}{2}]+1,$ let $p(x,\xi)$ be a symbol such that 
\begin{equation}|\partial_{x}^{\beta}p(x,\xi)|\leq C_{k},\,\,|\beta |\leq k. 
\end{equation}
Then $p(x,D):L^p(\mathbb{T}^n)\rightarrow L^q(\mathbb{T}^n)$ is a bounded operator, where $\frac{1}{p}+\frac{1}{q}=1.$
\end{lemma}
\begin{proof} By the definition of periodic pseudo-differential operator, integration by parts and inversion formula, we have

\begin{align*} \langle p(x,D_{x})u, g \rangle_{L^2(\mathbb{T}^n)} &=\int_{T^n}\sum_{\xi\in\mathbb{Z}^n}e^{i2\pi x\xi}p(x,\xi)\widehat{u}(\xi)g(x)dx\\
&=\int_{T^n}\sum_{\xi\in\mathbb{Z}^n}\sum_{\eta\in\mathbb{Z}^n}e^{i2\pi x(\xi-\eta)}p(x,\xi)\widehat{u}(\xi)\overline{\widehat{g}(\eta)}dx\\
&=\sum_{\xi\in\mathbb{Z}^n}\sum_{\eta\in\mathbb{Z}^n}\int_{T^n}e^{i2\pi x(\xi-\eta)}p(x,\xi)\widehat{u}(\xi)\overline{\widehat{g}(\eta)}dx\\
&=\sum_{\xi\in\mathbb{Z}^n}\sum_{\eta\in\mathbb{Z}^n}\int_{T^n}\langle \xi-\eta\rangle^{-2k}e^{i2\pi x(\xi-\eta)}L^{k}_{x}p(x,\xi)\widehat{u}(\xi)\overline{\widehat{g}(\eta)}dx,
\end{align*}

where $L^q_{x}=(I-\frac{1}{4\pi^2}\mathcal{L}_{x})^q$
and $\mathcal{L}_{x}$ is the Laplacian in $x-$variables. Using the Young Inequality, we get

\begin{align*} |\langle p(x,D_{x})u, g \rangle_{L^2(\mathbb{T}^n)}| &\leq \sum_{\xi\in\mathbb{Z}^n}\sum_{\eta\in\mathbb{Z}^n}\langle \xi-\eta\rangle^{-2k}|L^{k}_{x}p(x,\xi)||\widehat{u}(\xi)||\widehat{g}(\eta)|\\
&\leq \sum_{\xi\in\mathbb{Z}^n}\sum_{\eta\in\mathbb{Z}^n}\langle \xi-\eta\rangle^{-2k}C_{k}\widehat{u}(\xi)||\widehat{g}(\eta)|\\
&= \sum_{\xi\in\mathbb{Z}^n}|\widehat{u}(\xi)|\sum_{\eta\in\mathbb{Z}^n}\langle \xi-\eta\rangle^{-2k}|\widehat{g}(\eta)|C_{k}\\
&= \sum_{\xi\in\mathbb{Z}^n}C_{k}|\widehat{u}(\xi)|\langle \cdot\rangle^{-2q}*|\widehat{g}(\cdot)|(\xi)C_{k}\\
&\leq C_{k} \Vert \widehat{u}\Vert_{L^2(\mathbb{Z}^n)} \Vert\langle \cdot\rangle^{-2k}|\widehat{g}| \Vert_{L^2(\mathbb{Z}^n)}\\
&\leq C_{k}\Vert u\Vert_{L^2(\mathbb{T}^n)} \Vert\langle \cdot\rangle^{-2k} \Vert_{L^1(\mathbb{Z})}\Vert{g} \Vert_{L^2(\mathbb{T}^n)}
\end{align*}
If $p\geq 2$ then the inclusion map $i:L^p(\mathbb{T}^n)\rightarrow L^2(\mathbb{T}^n)$ is continuous. So, for some $C>0$ we have $\Vert \cdot\Vert_{L^2(\mathbb{T}^n)}\leq C\Vert \cdot\Vert_{L^p(\mathbb{T}^n)}.$
So, we get $$ |\langle p(x,D_{x})u, g \rangle_{L^2(\mathbb{T}^n)}| \leq CC_{k} \Vert\langle \cdot\rangle^{-2k} \Vert_{L^1(\mathbb{Z})} \Vert u \Vert_{L^p(\mathbb{T}^n)} \Vert g \Vert_{L^p(\mathbb{T}^n)}. $$ 
Finally, 
\begin{align*}\Vert p(x,D_{x})u \Vert_{L^q(\mathbb{T}^n)} &=\sup  \{ |\langle p(x,D_{x})u,g \rangle|:{\Vert g \Vert_{L^{p}(\mathbb{T}^n)}}\leq 1 \} \\
 &\leq CC_{k}\Vert\langle \cdot\rangle^{-2k} \Vert_{L^1(\mathbb{Z})} \Vert u \Vert_{L^p(\mathbb{T}^n)} .
\end{align*}
\end{proof}
By using the previous lemma we can prove the following theorem on $(L^p,L^r)$-boundedness of periodic operators.\\
\\
\noindent \textbf{Theorem II.}  { \it  Let $2\leq p<\infty.$ If $\sigma(x,\xi)$ satisfies the hypotheses of Theorem \ref{p1}, then  $\sigma(x,D):L^p(\mathbb{T}^n)\rightarrow L^r(\mathbb{T}^n)$ is a bounded linear operator for all $1<q\leq r\leq p<\infty,$ where $\frac{1}{p}+\frac{1}{q}=1.$}\vspace{0.2cm}
\begin{proof}
Theorem \ref{p1} implies that, $\sigma(x,D):L^p(\mathbb{T}^n)\rightarrow L^p(\mathbb{T}^n)$ is a bounded operator for $1<p<\infty.$ That $\sigma(x,D):L^p(\mathbb{T}^n)\rightarrow L^q(\mathbb{T}^n)$ is a bounded operator is a consequence of Lemma \ref{lemma22}. Now, by Riesz-Thorin interpolation theorem we deduce that $\sigma(x,D):L^p(\mathbb{T}^n)\rightarrow L^r(\mathbb{T}^n)$ is a bounded operator for all $q\leq r\leq p<\infty,$ with $2\leq p<\infty.$
\end{proof}
In the following theorem we analyze the boundedness of periodic amplitude  operators on $L^1.$\\
\\
\noindent \textbf{Theorem III.}  { Let $0<\varepsilon,\delta <1.$  If $a(x,y,D)$ is a positive amplitude operator with symbol satisfying the following inequalities
\begin{equation}
|\partial_{x}^{\beta}\partial_{y}^{\alpha}a(x,y,\xi)|\leq C_{\alpha,\beta}\langle \xi\rangle^{\delta|\beta|-\varepsilon|\alpha|},\,\,\,|\alpha|,|\beta|\leq[\frac{n}{2}(1-\delta)^{-1}]+1,
\end{equation}
then $a(x,y,D)$ is bounded on $L^1(\mathbb{T}^n).$ } \vspace{0.3cm}\\

\begin{proof} If  the operator $a(x,y,D)$ is positive and  $f\geq 0$ we have that $$a(x,y,D)f\geq 0.$$ Hence,
\begin{align*}|a & (x,y,D)f(x)|=a(x,y,D)f(x)\\
&=\int\limits_{\mathbb{T}^n \times \mathbb{R}^n}e^{i2\pi (x-y)\xi}a(x,y,\xi)f(y)d\xi\,dy\\
&=\int\limits_{\mathbb{T}^n\times \mathbb{R}^n\times \mathbb{R}^n}e^{i2\pi (x-y)\xi}e^{i2\pi(\eta)y}a(x,y,\xi)d\xi\,dy\widehat{f}(\eta)d\eta\\
&=\int\limits_{\mathbb{T}^n\times \mathbb{R}^n\times \mathbb{R}^n}e^{i2\pi (x)\xi}e^{i2\pi(\eta-\xi)y}a(x,y,\xi)d\xi\,dy\widehat{f}(\eta)d\eta.
\end{align*}
\noindent Clearly $\langle \xi \rangle^{-2q}L^q_{x}e^{i2\pi x\xi}=e^{i2\pi x\xi},$ where $L^q_{x}=(I-\frac{1}{4\pi^2}\mathcal{L}_{x})^q$
and $\mathcal{L}_{x}$ is the Laplacian in $x-$variables. Using integration by parts successively we obtain,  

\begin{align*}\int\limits_{\mathbb{T}^n}| & a(x,y,D)f(x)|dx=\int\limits_{\mathbb{T}^n\times \mathbb{T}^n \times \mathbb{R}^n\times \mathbb{R}^n}e^{i2\pi (x)\xi}e^{i2\pi(\eta-\xi)y}a(x,y,\xi)d\xi\,dy\widehat{f}(\eta)d\eta\,dx\\
&=\int\limits_{\mathbb{T}^n\times \mathbb{T}^n \times \mathbb{R}^n\times \mathbb{R}^n}\langle \xi \rangle^{-2q}\langle \eta-\xi \rangle^{-2q} e^{i2\pi (x)\xi}e^{i2\pi(\eta-\xi)y} L^q_{x}L^{q}_{y}a(x,y,\xi)d\xi\,dy\widehat{f}(\eta)d\eta\,dx \\
&\leq \int\limits_{\mathbb{T}^n\times \mathbb{T}^n \times \mathbb{R}^n\times \mathbb{R}^n}\langle \xi \rangle^{-2q}\langle \eta-\xi \rangle^{-2q}  |L^q_{x}L^{q}_{y}a(x,y,\xi)|d\xi\,dy|\widehat{f}(\eta)|d\eta\,dx\\
&\leq \int\limits_{\mathbb{T}^n\times \mathbb{T}^n \times \mathbb{R}^n\times \mathbb{R}^n}\langle \xi \rangle^{-2q}\langle \eta-\xi \rangle^{-2q}  C_{q}\langle \xi\rangle^{\delta\cdot 2q} d\xi\,dy|\widehat{f}(\eta)|d\eta\,dx\\
&\leq \int\limits_{\mathbb{T}^n\times \mathbb{T}^n \times \mathbb{R}^n\times \mathbb{R}^n}\langle \xi \rangle^{2q(\delta-1)}\langle \eta-\xi \rangle^{-2q}  C_{q} d\xi\,dy\Vert f\Vert_{L^1(\mathbb{T}^n)} d\eta\,dx\\
&\leq C \Vert f\Vert_{L^1(\mathbb{T}^n)},\\
\end{align*} where $C=\int\limits_{\mathbb{R}^n} (\langle \cdot \rangle^{2q(\delta-1)}*\langle \cdot \rangle^{-2q} )(\eta) C_{q} d\eta.$ Clearly $C$ is finite for $\frac{n}{2}(1-\delta)^{-1}<q\leq[\frac{n}{2}(1-\delta)^{-1}]+1$.

Our proof of the $L^1-$boundedness is for non-negative $f,$ but this is sufficient since an arbitrary real function can be decomposed into its positive and negative parts, and complex functions into its real and imaginary parts.
\end{proof}

\begin{remark}
Positive operators in the form of Theorem $\textnormal{IV}$ arise with Bessel potentials of order $m,$ $m\in \mathbb{R}.$ We recall that for every $m\in\mathbb{R},$ the Bessel potential of order $m,$ denoted by $\langle D_{x}\rangle^m$ is the pseudo-differential operator with symbol $\sigma_{m}(\xi)=\langle \xi\rangle^m,$ $\xi\in \mathbb{Z}^n.$ 
\end{remark}

\begin{remark} There exists a connection between the $L^{p}$ boundedness of Fourier multipliers on compact Lie groups and its continuity on Besov spaces $B^r_{p,q}$. This fact was proved by the author in Theorem 1.2 of \cite{Duvan5}. In fact, the Lie group structure of the torus $\mathbb{T}^n$ implies that every periodic Fourier multiplier (i.e. a periodic pseudo-differential operator with symbol depending only on the dual variable $\xi\in\mathbb{Z}^n$) bounded from $L^{p_{1}}$ into $L^{p_{2}}$ also is bounded from $B^{r}_{p_1,q}$ into $B^{r}_{p_2,q},$ $r\in\mathbb{R}$ and $0<q\leq \infty.$ So, restrictions of our results to the case of symbols $\sigma(\xi)$ associated to Fourier multipliers give the boundedness of these operators on every Besov space $B^r_{p,q}(\mathbb{T}^n),$ $1<p<\infty. $
\end{remark}
\noindent \textbf{Acknowledgments:} I would like to thank the anonymous referee for his/her remarks which helped to improve the manuscript. This project was partially supported by Universidad de los Andes, Mathematics Department, Bogot\'a-Colombia.

\bibliographystyle{amsplain}

\end{document}